\documentclass[10pt,a4paper,oneside]{article}
\usepackage{amsmath, graphicx, amsthm,amssymb}
\graphicspath{{C:/Users/Eric/Desktop/LaTex/}}
\newtheorem{theorem}{Theorem}[section]
\newtheorem*{theorem*}{Theorem}

\newtheorem{definition}{Definition}[section]
\numberwithin{equation}{section}
\begin{document}
\title{{\bf Senior Seminar II}\vspace{5mm}\\On the Oscillations of  Second Order Linear Differential Equations}
\date{Spring 2013}
\author{Eric Robert Kehoe\\University of Massachusetts: Lowell}
\maketitle

\newpage

\section{Introduction to Research}
\subsection{The Discriminant}
This paper examines the properties of second order linear differential equations, in particular the oscillatory behavior of solutions to these equations \cite{1,2,3}. The research started from a mistake in utilizing the well known characteristic equation, used in solving constant coefficient second order equations, to predict oscillations in solutions to variable coefficient second order equations, i.e.

\begin{equation}
y''(x) +  b(x)y'(x) + c(x)y(x) = 0 \label{ODE1}\\
\end{equation}

\noindent where the characteristic polynomial associated with this equation is,

\begin{equation}
m^{2}(x) + b(x)m(x) + c(x) = 0 \label{charPoly1}\\
\end{equation}

\noindent We also know from elementary algebra that, for every fixed $x_0\in\mathbb{R}$ such that \eqref{charPoly1} is well defined, we will have complex roots when,\\

\begin{equation}
 b^{2}(x_0) - 4c(x_0) < 0 \label{desc}
\end{equation}

\noindent Implying that $m(x) = \alpha(x) \pm i\beta(x)$ where $\alpha(x)$ and $\beta(x)$ are real valued functions. The term on the left hand side of the inequality above is called the discriminant of the characteristic equation. In the case where $b(x)$ and $c(x)$ are constant this implies $\alpha(x)$ and $\beta(x)$ are constant and a general solution of,

\begin{equation}
y(x) = \mathrm{e}^{\alpha x}(A\sin(\beta x )+B\cos(\beta x)) \notag
\end{equation}
 
\noindent for \eqref{ODE1} where $A$ and $B$ are constants . Since these solutions oscillate when $m$ is complex, this motivated the use of the discriminant to predict where solutions to \eqref{ODE1} will oscillate, assuming $b(x)$ and $c(x)$ are non-constant. Intuitively, one is tempted to think that if at every fixed $x_0$ in an interval $I$  \eqref{desc} is satisfied, then the actual solutions to \eqref{ODE1} will oscillate on that interval based on the oscillating behavior of the local solutions,  $y(x) = \mathrm{e}^{\alpha(x_0) x}(A\sin(\beta(x_0) x )+B\cos(\beta(x_0) x)$,  at every point in the interval.\\\\
\noindent Later we will find out that this is the wrong intuition. Nevertheless, the discriminant sucessfully predicted oscillatory behavior in solutions to the Parabolic Cylinder,the Bessel, the Hermite, and the Airy differential equations. i.e.

\begin{equation}
\begin{aligned}
y'' + (\frac{1}{4}x^{2}-a)y &= 0 &&\text{(Parabolic Cylinder)}\\ \notag
y'' -xy &= 0 &&\text{(Airy)}\\
y''+\frac{1}{x}y'+(1-\frac{n^2}{x^2})&= 0 &&\text{(Bessel)}\\ 
y''-2xy'+2\lambda y &= 0 &&\text{(Hermite)}\\
\end{aligned}
\end{equation}
\noindent For the Parabolic Cylinder we have,
\begin{equation}
D = -4(\frac{1}{4}x^{2}-a) < 0 \Leftrightarrow x < -2\sqrt{a} \text{\hspace{3mm} or\hspace{3mm}} x > 2\sqrt{a} \notag
\end{equation}

\noindent By taking $a = 16$, we have $x < -8$ or $x > 8$\\
\vspace{1mm}\\
\centerline{\includegraphics[scale=.2]{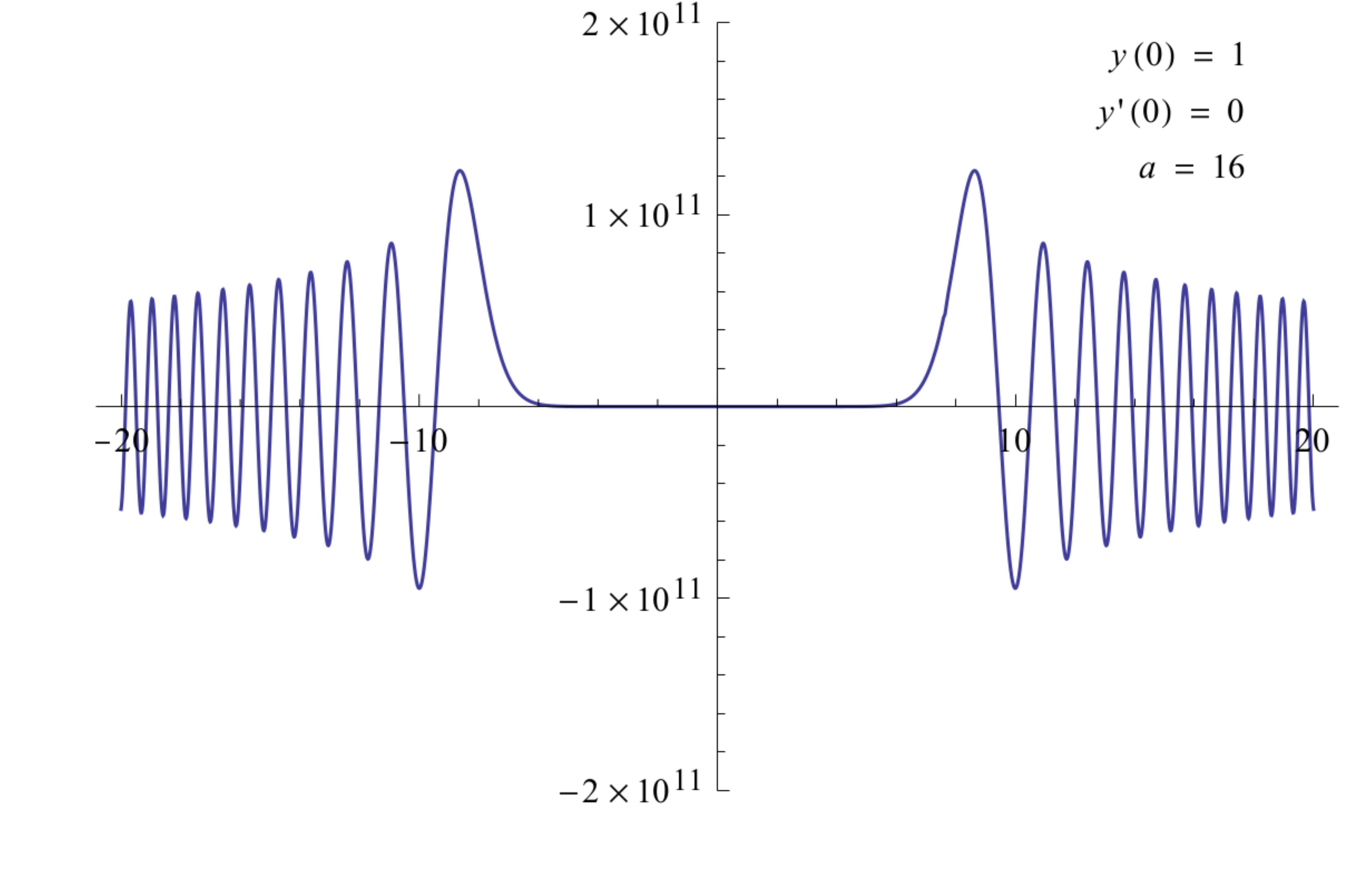}}
\vspace{3mm}\\

\noindent For the Airy we have,
\begin{equation}
D = -4(-x) < 0 \Leftrightarrow x < 0 \notag
\end{equation}\\
\vspace{1mm}\\
\centerline{\includegraphics[scale=.2]{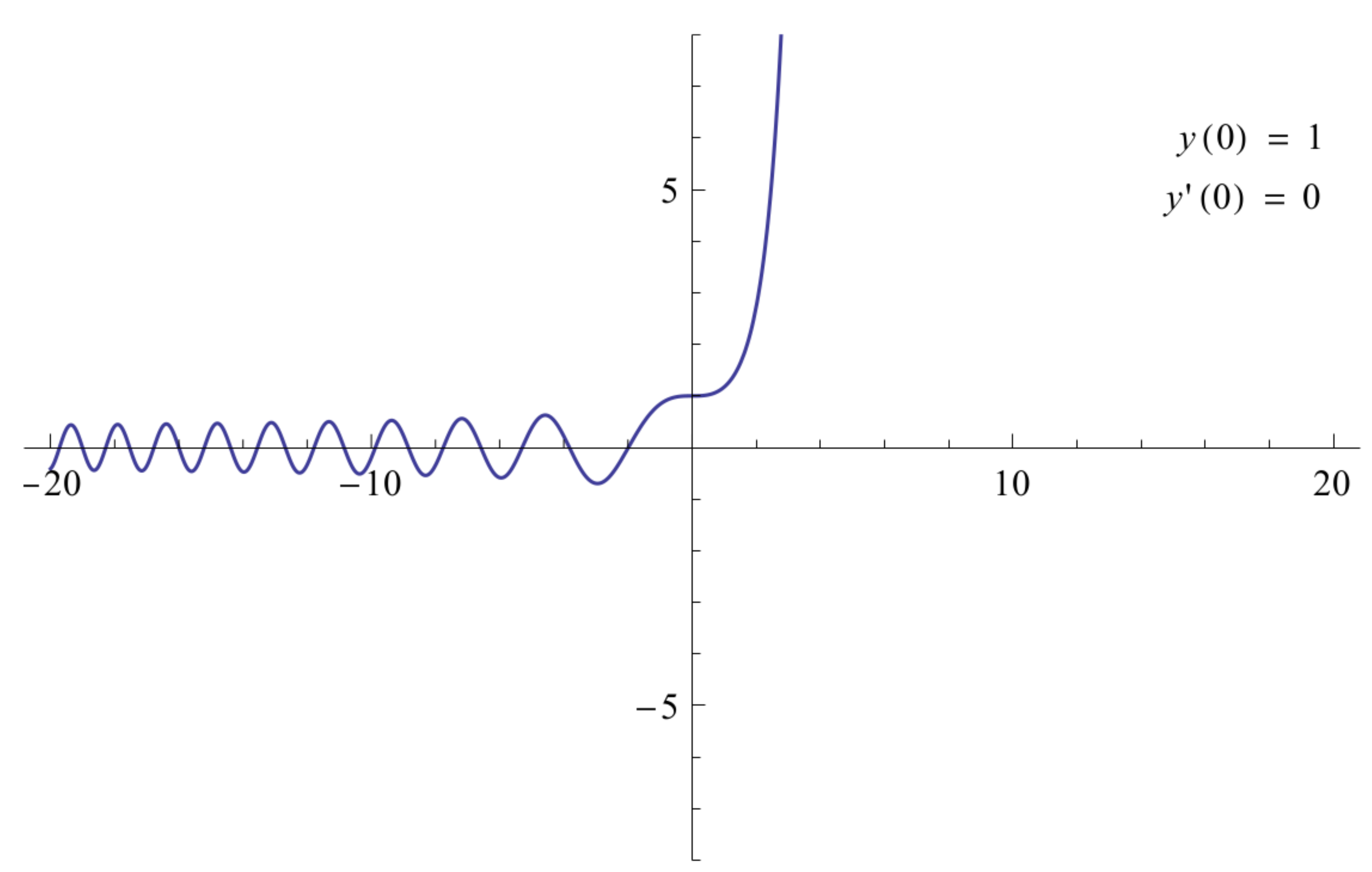}}
\newpage
\noindent For the Bessel we have,
\begin{equation}
D = \left(\frac{1}{x}\right)^{2}-4\left(1-\frac{n^{2}}{x^{2}}\right) < 0 \Leftrightarrow  x < -\sqrt{n^{2}+\frac{1}{4}} \text{\hspace{3mm} or\hspace{3mm}} x > \sqrt{n^{2}+\frac{1}{4}} \notag
\end{equation}\\
\vspace{1mm}\\
\centerline{\includegraphics[scale=.4]{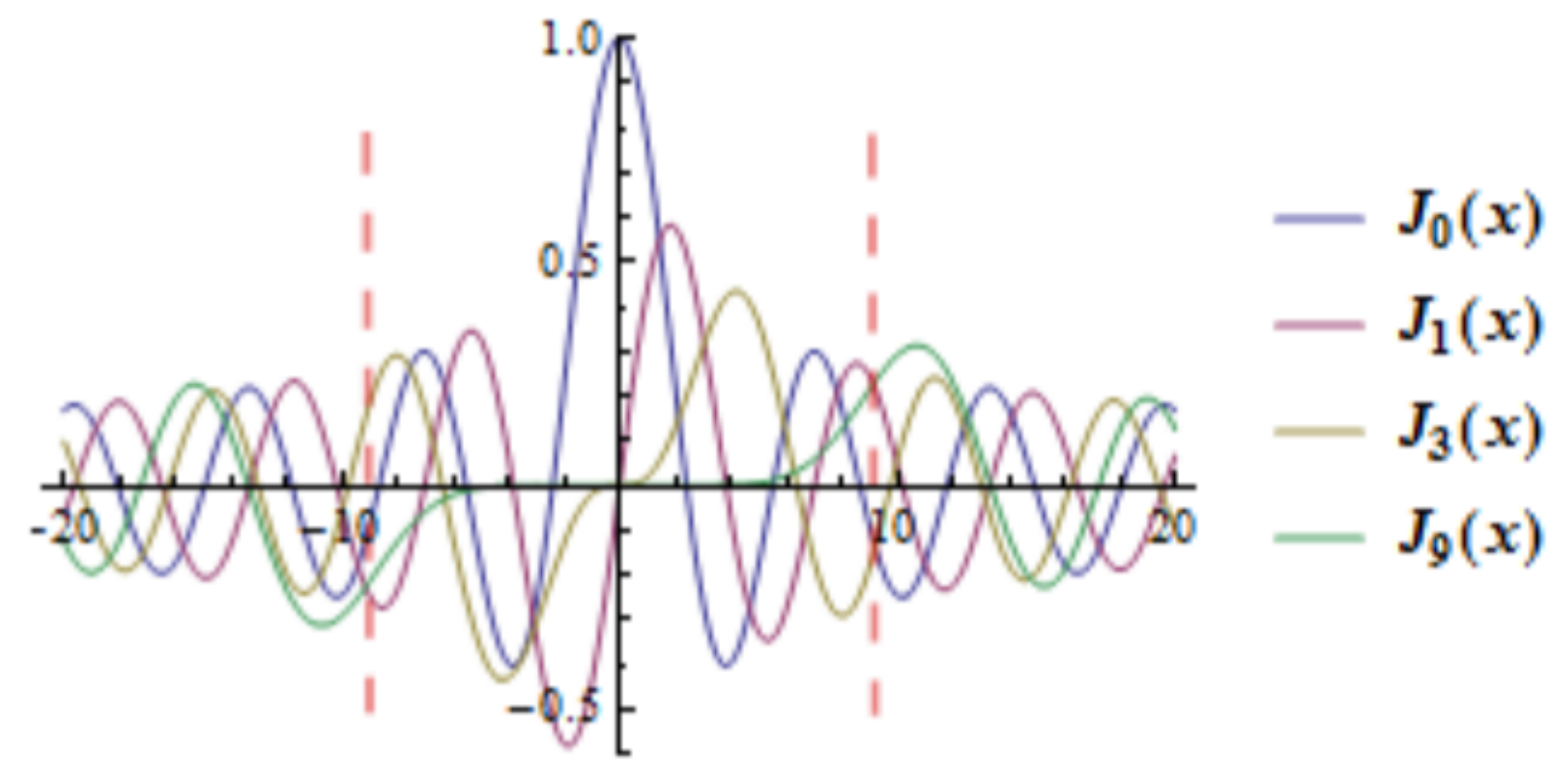}}\\
\vspace{1mm}\\
In the case of $J_9(x)$, $D < 0$ when $x < -\frac{5\sqrt{13}}{2}$ or $x > \frac{5\sqrt{13}}{2}$, which is indicated by the red verticals in the figure above.\\
\vspace{1mm}\\
\noindent For the Hermite we have,
\begin{equation}
D =(-2x)^{2}-4(2\lambda)<0 \Leftrightarrow -\sqrt{2\lambda} < x < \sqrt{2\lambda} \notag
\end{equation}\\
By taking $\lambda = 18$, we have $-6 < x < 6$.\\
\vspace{1mm}\\
\centerline{\includegraphics[scale=.2]{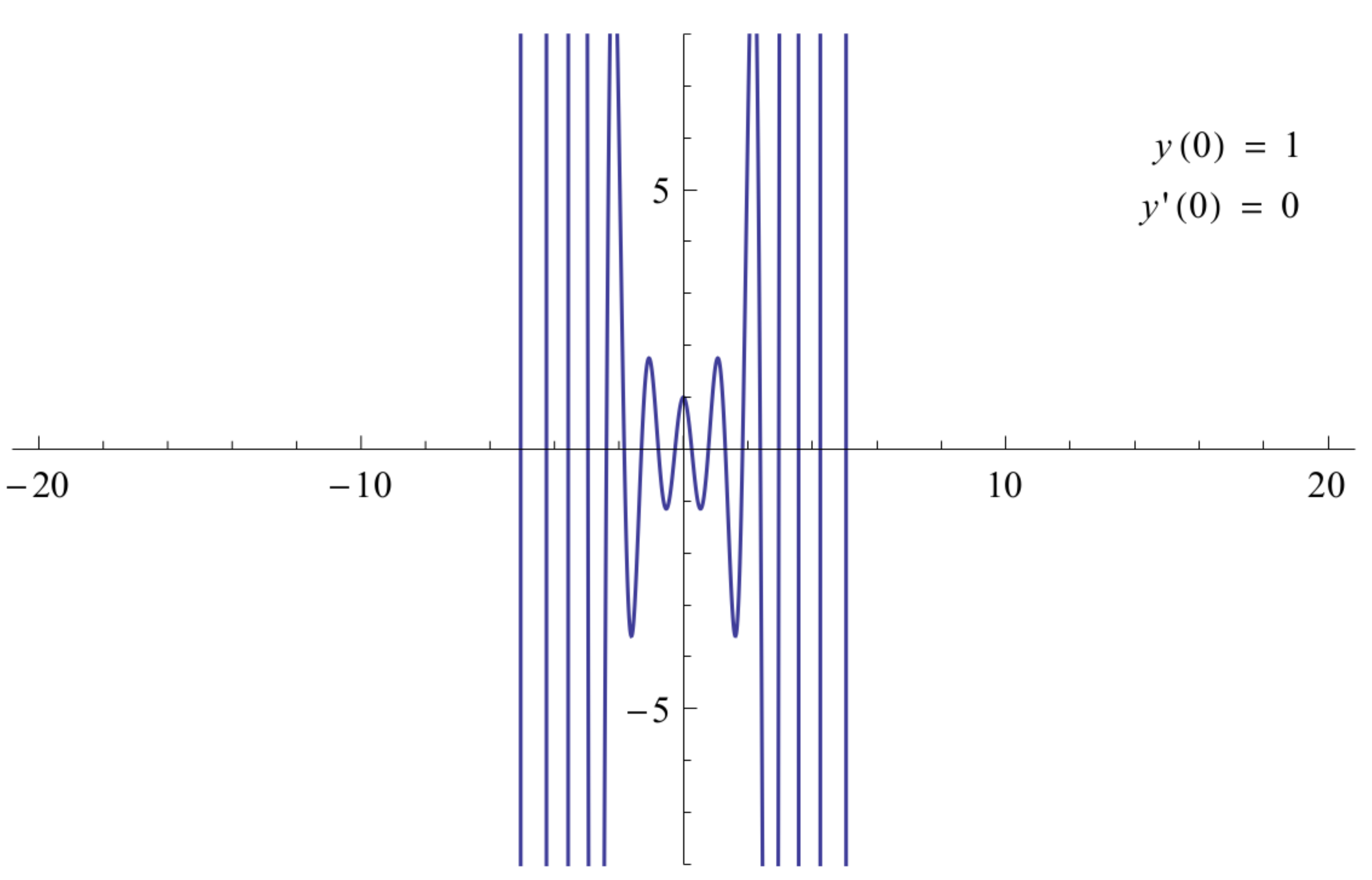}}\\
\vspace{1mm}\\
\noindent In each example above, the discriminant correctly predicts where the solutions will oscillate. These empircal findings motivated the generalization of the\\ Euler transformation, used to produce the characteristic equation in the case of constant coefficients, in order to prove the validity of the discriminant.

\subsection{The Riccati Equation}
Consider the differential equation,
\begin{equation}
y''(x) +  b(x)y'(x) + c(x)y(x) = 0 \label{ODE2}\\
\end{equation}
Then let $y(x) = \mathrm{e}^{\int{m}(x)\mathrm{d}x}$. We have,
\begin{equation}
\begin{aligned}
y'(x) &= m(x)\mathrm{e}^{\int{m}(x)\mathrm{d}x}\\
y''(x) &= m'(x)\mathrm{e}^{\int{m}(x)\mathrm{d}x}+m^{2}(x)\mathrm{e}^{\int{m}(x)\mathrm{d}x}\notag
\end{aligned}
\end{equation}
\noindent Substituing this into \eqref{ODE2} yields,
\begin{equation}
\begin{aligned}
\mathrm{e}^{\int{m}(x)\mathrm{d}x}\left(m'(x) +  m^{2}(x) + b(x)m(x) + c(x)\right) = 0&\\
\vspace{1mm}\\
\Rightarrow -\frac{dm}{dx} = m^{2}(x) + b(x)m(x) + c(x)&\\ \label{RicDE}
\end{aligned}
\end{equation}
Since $\forall x\in\mathbb{R}$ we have $\mathrm{e}^{\int{m}(x)\mathrm{d}x} \not= 0$.\\
\vspace{1mm}\\
This equation is called a scalar Riccati equation. It is a non-linear first order differential equation. In the case of constant coefficients, m is constant and equation \eqref{RicDE} becomes the well known characteristc equation. In the case that $b^{2}(x)-4c(x)<0$, then \eqref{RicDE} can be factored as such,
\begin{equation}
 -\frac{dm}{dx} = \left(m(x) + \frac{b(x)}{2}\right)^{2}+\left(\frac{\sqrt{4c(x)-b^{2}(x)}}{2}\right)^{2} \notag
\end{equation}
Although interesting, the Riccati equation did not yield the reasoning behind the predictive nature of the discriminant.

\section{Oscillation on Infinite Intervals}
We begin this section with a counter example to the discriminant. Consider the equation,
\begin{equation}
y''+\frac{1}{x}y'+\frac{k^{2}}{x^{2}}y=0 \notag
\end{equation}
\noindent This equation yields a particular solution of $y=\sin(k\ln x)$, which clearly oscillates by any sensical definition of oscillating. Yet for $k<\frac{1}{2}$ we have,
\begin{equation}
 D=\frac{1-4k^{2}}{x^{2}}>0 \notag
\end{equation}
\noindent This counter example motivated a different definition of the discriminant for predicting when solutions of \eqref{ODE1} will oscillate and not oscillate. The new discriminant was found via transformation of \eqref{ODE1} into normal form.
\begin{definition}
Let $\psi(x)$ be a twice differential non-zero real-valued function, we say that $\psi(x)$ oscillates on an infinite interval $I\subseteq\mathbb{R}$ if $\psi(x)$ vanishes infinitely many times on $I$.\\
\vspace{1mm}\\
Note: An infinite interval is defined as any interval of the form $I = [a,\infty)$ or $I = (-\infty,a]$ or $I=(-\infty,\infty)$ where $a\in\mathbb{R}$.
\end{definition}
\noindent Consider the differential equation,
\begin{equation}
u''(x) + Q(x)u(x) = 0 \label{nODE}
\end{equation}
\noindent where $Q(x)$ is a real-valued function, continuous on it's domain. The differential equation above is said to be in {\bf normal form}, and in fact any differential equation of the general form in \eqref{ODE1} can be put in normal form by a simple transformation.\\
\vspace{1mm}\\
\noindent Let $y = u\mathrm{e}^{-\frac{1}{2}\int{b\ \mathrm{d}x}}$ in \eqref{ODE1}. Then,
\begin{equation}
\begin{aligned}
y' &= \mathrm{e}^{-\frac{1}{2}\int{b\ \mathrm{d}x}}\left(u'-\frac{1}{2}bu\right)\\ \notag
y'' &= \mathrm{e}^{-\frac{1}{2}\int{b\ \mathrm{d}x}}\left(u''-bu'-\frac{1}{2}b'u+\frac{1}{4}b^{2}u\right)\\
\end{aligned}
\end{equation}
\begin{equation}
\begin{aligned}
&\Rightarrow\  \mathrm{e}^{-\frac{1}{2}\int{b\ \mathrm{d}x}}\left[u''-bu'-\frac{1}{2}b'u+\frac{1}{4}b^{2}u+b\left(u'-\frac{1}{2}bu\right)+cu\right]=0\\
&\Rightarrow\ u''-\frac{1}{2}b'u-\frac{1}{4}b^{2}u+cu=0\  \Rightarrow\ u''-\frac{1}{4}\left(b^{2}-4c+2b'\right)u=0\\
\vspace{1mm}\\
&\Rightarrow\ u''+Qu=0\\ \notag
\end{aligned}
\end{equation}
\noindent where $Q(x)= -\frac{1}{4}\left(b^{2}(x)-4c(x)+2b'(x)\right)$

\begin{theorem}{\bf{Sturm's Comparison  Theorem} \cite{1}}
Let $\psi_1(x)$ and $\psi_2(x)$ be two non-trivial solutions to,
\begin{equation}
u''+Q_1(x)u=0\ \ \ \text{and}\ \ \ u''+Q_2(x)u=0 \notag
\end{equation}
respectively, on an interval I. If $Q_1$ and $Q_2$ are continuous functions such that $Q_1>Q_2$ on I, then between any two consecutive zeros $x_1$ and $x_2$ of $\psi_2$, $\psi_1$ vanishes at least once.
\end{theorem}
\begin{proof}
Let $x_1$ and $x_2$ be two consecutive zeroes of $\psi_2$. Without loss of generality, assume $\psi_1>0$ and $\psi_2>0$ on $(x_1,x_2)$. We then have that $\psi_2'(x_1)\geq 0$ and $\psi_2'(x_2)\leq 0$. Since  $\psi_1>0$ on $(x_1,x_2)$ then we also have that $\psi_1(x_1)\geq 0$ and $\psi_1(x_2)\geq 0$ This implies the Wronskian, defined as $W=\psi_1\psi_2'-\psi_2\psi_1'$, is $W(x_1)=\psi_1(x_1)\psi_2'(x_1)\geq 0$ and $W(x_2)=\psi_1(x_2)\psi_2'(x_2)\leq 0$. This yields the inequality, $W(x_1)\geq W(x_2)$.\\
\vspace{1mm}\\
\centerline{\includegraphics[scale=.2]{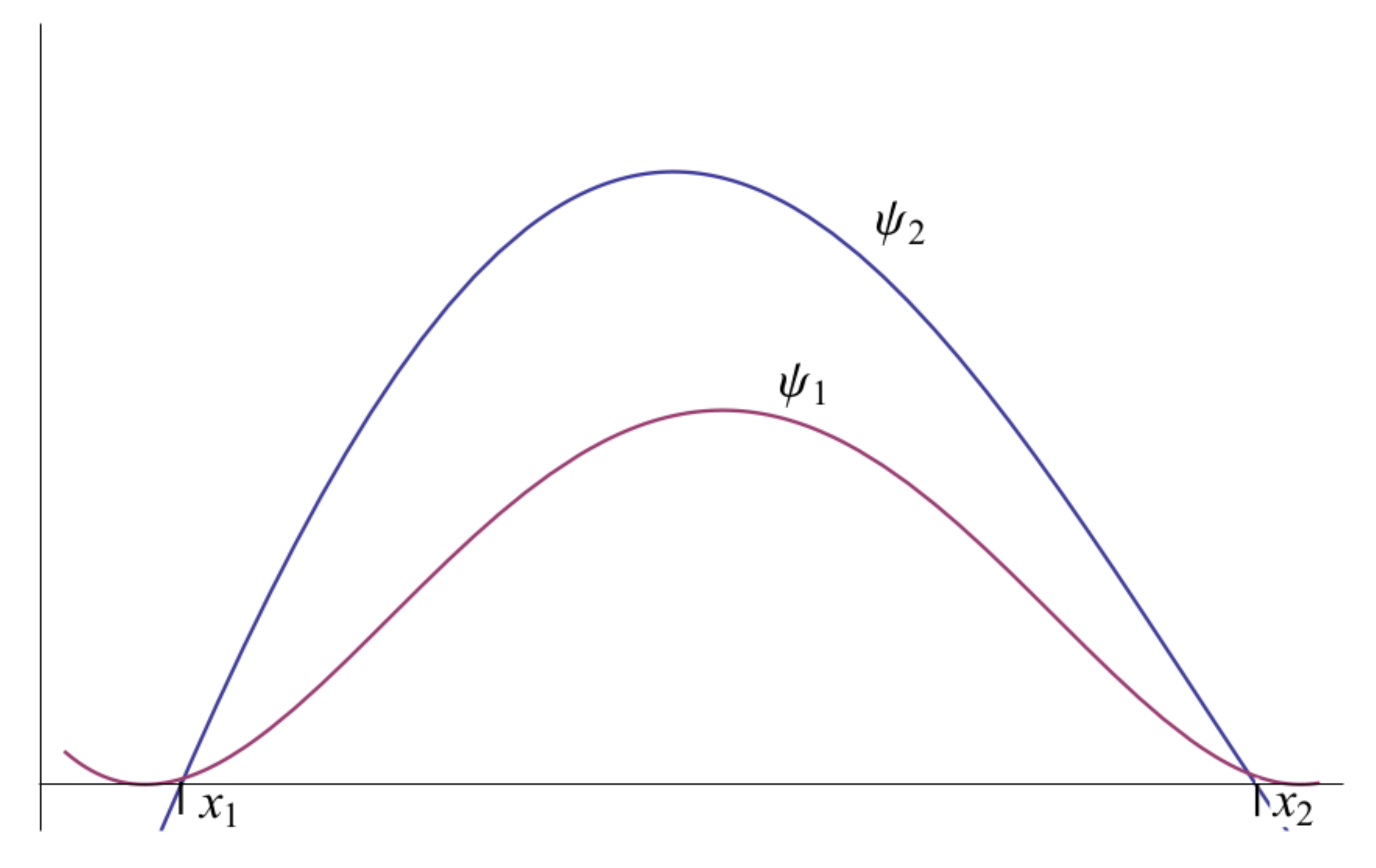}}\\
\vspace{1mm}\\
We then have $W'=\psi_1\psi_2''-\psi_2\psi_1''=(Q_1-Q_2)\psi_1\psi_2>0$ on $(x_1,x_2)$. This implies that $W$ is strictly increasing on $(x_1,x_2)$, so that $W(x_1)<W(x_2)$. But this contradicts  $W(x_1)\geq W(x_2)$.\\
\qedhere
\end{proof}
\begin{theorem}
Consider the differential equation,
\begin{equation}
y''(x) +  b(x)y'(x) + c(x)y(x) = 0 \label{ODE3}\\
\end{equation}
If $Q=-\frac{1}{4}\left(b^{2}-4c+2b'\right)$ is bounded below by a positive number on an infinite interval $I$, then a non-trivial solution, $\psi(x)$, to \eqref{ODE3} oscillates on I.
\end{theorem}
\begin{proof}
Let $\psi(x)$ be a non-trivial solution of \eqref{ODE3}. Since $Q$ is bounded below by a positive number on $I$,  there exists an $\varepsilon>0$ such that $Q>\varepsilon^{2}>0$ on $I$. By writing \eqref{ODE3} in normal form we have,
\begin{equation}
 u''-\frac{1}{4}\left(b^{2}-4c+2b'\right)u=0 \label{Normal}
\end{equation}
We compare this to the differential equation,
\begin{equation}
u''+\varepsilon^{2}u=0 \label{Compare}
\end{equation}
By transformation we have that $\xi_1(x)=\psi(x)\mathrm{e}^{\frac{1}{2}\int{b(x)\ \mathrm{d}x}}$ is a solution of \eqref{Normal}. Sturm's Comparison Theorem implies that $\xi_1(x)$ vanishes at least once between consecutive zeros of $\xi_2(x)=sin(\varepsilon x)$ on $I$. Since $I$ is an infinite interval we have that $\xi_2(x)$ vanishes infinitely many times on I, and it follows that $\xi_1(x)$ vanishes infinitely many times on I. Since $\xi_1(x)$ is a solution of \eqref{Normal} and $\mathrm{e}^{-\frac{1}{2}\int{b(x)\ \mathrm{d}x}}>0$, we have that $\psi(x)=\xi_1(x)\mathrm{e}^{-\frac{1}{2}\int{b(x)\ \mathrm{d}x}}$ vanishes where ever $\xi_1(x)$ vanishes. Thus concluding $\psi(x)$ oscillates on $I$.\\
\qedhere
\end{proof}
\noindent Note: It is not enough that $Q>0$ ; it should always be verified that $Q$ is bounded below by a positive number in order to guarentee oscillations of solutions. Consider the example below,
\begin{equation}
y''+\frac{1}{4x^{2}}y=0 \notag
\end{equation}
In this equation $Q(x)=\frac{1}{4x^{2}}>0$ for all x. But $Q$ is not bounded below by a positive number; $Q$ approaches zero as x approaches infinity. This equation has an elementary solution of $\psi(x)=\sqrt{x}$, which clearly does not oscillate on any interval.

\subsection{Examples}
We now refer back to the examples in section 1 and prove the oscillatory behavior for solutions to the Parabolic Cylinder, Airy and Bessel type differential equations. Let $\varepsilon>0$. For the Parabolic Cylinder we have,
\begin{equation}
\begin{aligned}
&y'' +\left(\frac{1}{4}x^{2}-a\right)y = 0\\ \notag
&\Rightarrow Q(x)= -\frac{1}{4}\left(-4\left( \frac{1}{4}x^{2}-a \right)\right)\\
&\Leftrightarrow Q(x)= \left( \frac{1}{4}x^{2}-a \right)\\
&\Leftrightarrow Q(x)\geq\varepsilon\Leftrightarrow  x \leq -2\sqrt{a+\varepsilon} \text{\hspace{3mm} or\hspace{3mm}} x \geq 2\sqrt{a+\varepsilon} \\
\end{aligned}
\end{equation}
Thus, any non-trivial solution to the Parabolic Cylinder D.E. will oscillate on the infinite intervals $I_1=(-\infty,-2\sqrt{a+\varepsilon}]$ and $I_2=[2\sqrt{a+\varepsilon},\infty)$. Since this condition holds for any $\varepsilon>0$, and oscillating is a global condition, we agree to say simply that the solutions oscillate on  $I_1=(-\infty,-2\sqrt{a}]$  and $I_2=[2\sqrt{a},\infty)$.\\
\vspace{1mm}\\
For the Airy we have,
\begin{equation}
\begin{aligned}
&y'' -xy = 0\\ \notag
&\Rightarrow Q(x)= -\frac{1}{4}\left(-4\left(-x \right)\right)\\
&\Leftrightarrow Q(x)=-x\\
&\Leftrightarrow Q(x)\geq\varepsilon\Leftrightarrow  x \leq -\varepsilon\\
\end{aligned}
\end{equation}
Thus, any non-trivial solution to the Airy D.E. will oscillate on the infinite interval $I=(-\infty,-\varepsilon]$. Again since this holds for any $\varepsilon>0$, we say solutions to the Airy D.E oscillate on the interval  $I=(-\infty,0]$.\\
Lastly, for the Bessel we have,
\begin{equation}
\begin{aligned}
&y''+\frac{1}{x}y'+(1-\frac{n^2}{x^2})y= 0\\ \notag
&\Rightarrow Q(x)= -\frac{1}{4}\left(\frac{1}{x^{2}}-4\left(1-\frac{n^2}{x^2}\right)-\frac{2}{x^{2}}\right)\\
&\Leftrightarrow Q(x)= -\frac{1}{4}\left(-4\left(1-\frac{n^2}{x^2}\right)-\frac{1}{x^{2}}\right)\\
&\Leftrightarrow Q(x)= -\frac{1}{4}\left(-4+\frac{4n^2}{x^2}-\frac{1}{x^{2}}\right)\\
&\Leftrightarrow Q(x)= -\frac{1}{4}\left(-4+\frac{4n^2-1}{x^2}\right)\\
&\Leftrightarrow Q(x)\geq\frac{1}{4}\varepsilon \Leftrightarrow  \left(-4+\frac{4n^2-1}{x^2}\right) \leq -\varepsilon\\
&\Leftrightarrow  x^{2}\geq\frac{4n^2-1}{ 4-\varepsilon}\\
&\Leftrightarrow  x \leq -\sqrt{\frac{4n^2-1}{ 4-\varepsilon}} \text{\hspace{3mm} or\hspace{3mm}} x \geq  \sqrt{\frac{4n^2-1}{ 4-\varepsilon}}
\end{aligned}
\end{equation}
Thus, any non-trivial solution to the Bessel D.E. will oscillate on the infinite intervals $I_1=(-\infty, -\sqrt{\frac{4n^2-1}{ 4-\varepsilon}}]$ and $I_2=[\sqrt{\frac{4n^2-1}{ 4-\varepsilon}},\infty)$. Since this holds for any $\varepsilon>0$, we say solutions to the Bessel D.E oscillate on the intervals $I_1=(-\infty, -\sqrt{n^{2}-\frac{1}{4}}]$ and $I_2=[\sqrt{n^{2}-\frac{1}{4}},\infty)$. This disagrees with the original prediction of the discriminant by an additive factor of one-half under the root.\\
\subsection{Extension to Non-Homogenous Equations}
Consider the differential equation,
\begin{equation}
y''(x) +  b(x)y'(x) + c(x)y(x) = F(x) \label{Non}\\
\end{equation}
Suppose $Q=-\frac{1}{4}\left(b^{2}-4c+2b'\right)$ is bounded below by a positive number on an infinte interval $I$, then let  $\psi_g(x)$ be the general solution to \eqref{Non}. The general solution $\psi_g(x)$ can be represented as  $\psi_g(x)=C_1\psi_1(x)+C_2\psi_2(x)+\psi_p(x)$. Where $\psi_1(x)$ and $\psi_2(x)$ are the linearly independent solutions to the corresponding homogeneous equation, and $\psi_p(x)$ is the particular solution to \eqref{Non}. Thus if $\psi(x)$ is a solution of \eqref{Non}, and $C_1$ and $C_2$ are not both zero, then $\psi(x)-\psi_p(x)$ is a non-trivial solution of \eqref{ODE3}. By Theorem 2.2 we have that $\psi(x)-\psi_p(x)$ oscillates on $I$.\\
\begin{definition}
Let $\psi(x)$ be a solution to the differential equation,
\begin{equation}
y''(x) +  b(x)y'(x) + c(x)y(x) = F(x) \notag
\end{equation}
we say that $\psi(x)$ oscillates about the particular solution $\psi_p(x)$ on an infinite interval $I$ if $\psi(x)-\psi_p(x)$ oscillates on $I$.
\end{definition}
\noindent As an example, consider the differential equation,
\begin{equation}
y'' -xy = 5x^{2} \notag
\end{equation}
This a non-homogeneous Airy type differential equation and in the figure below one can see that the solution to this equation oscillates about the particular solution $\psi_p(x)=-5x$.
\vspace{5mm}\\
\centerline{\includegraphics[scale=.2]{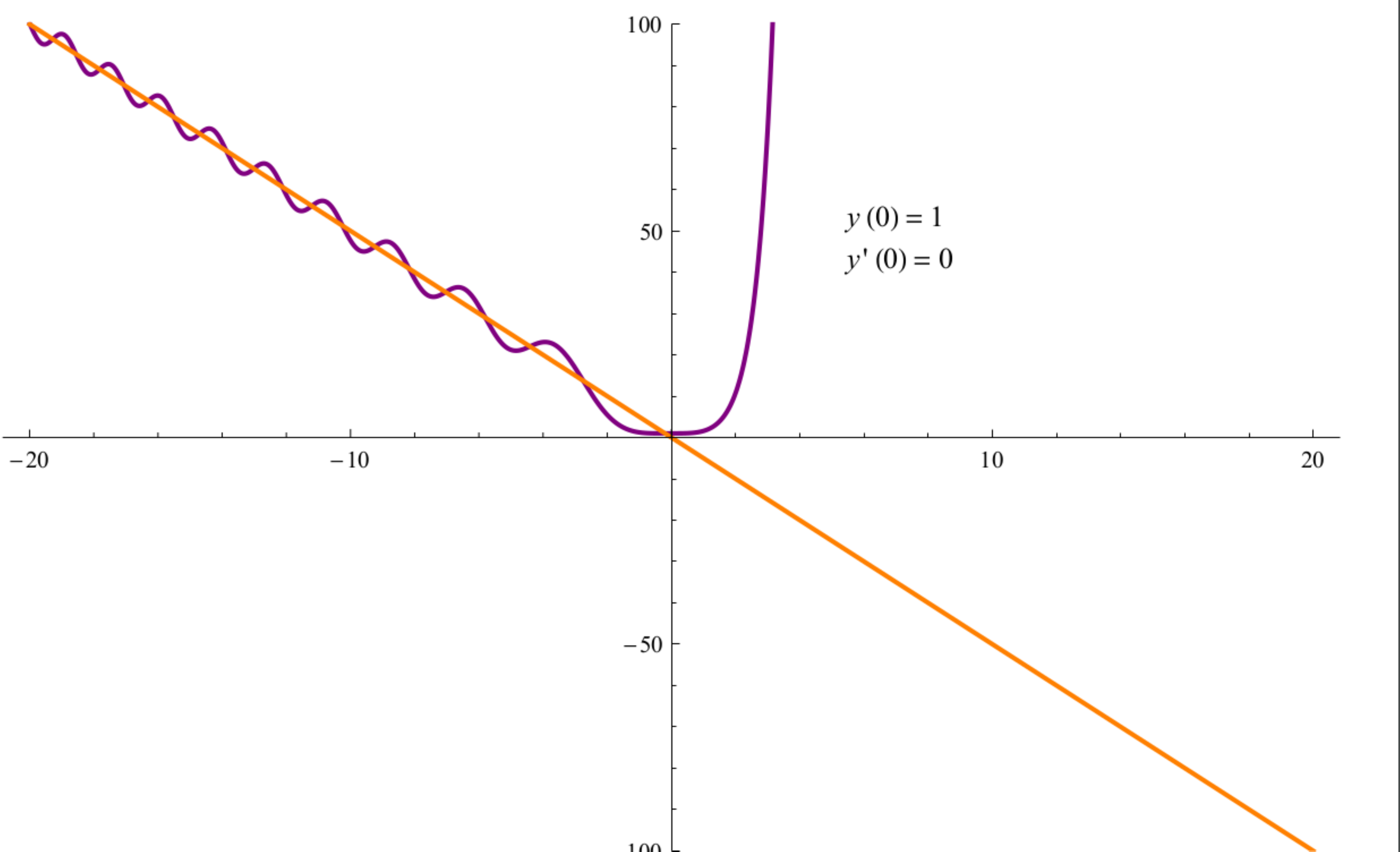}}\\
\vspace{1mm}\\

\section{The Discriminant Test}
So far we have looked at the conditions for when solutions to \eqref{ODE1} oscillate on infinite intervals via Sturm's Comparision Theorem. The original test for oscillations which motivated this research although, was the decriminant of \eqref{ODE1} discussed in Section 1. This discriminant was found to be impercise in predicting oscillations for solutions to \eqref{ODE1}  with a non-zero $b$ term. After further research and the use of Sturm's Theorem, it became apparent that the true discriminant in predicting oscillations was embedded within the $Q$ from the differential equation's normal form.

\begin{definition}
\noindent Consider the differential equation,
\begin{equation}
y''(x)+b(x)y'(x)+c(x)y(x) = 0 \label{ODE3}
\end{equation}
The  function $D(x)=b^{2}(x)-4c(x)+2b'(x)$ is called the  {\bf discriminant} of the differential equation.
\end{definition}
\begin{definition}
Let $\psi(x)$ be a twice differential non-zero real-valued function, we say that $\psi(x)$ does not oscillate on an interval $I\subseteq\mathbb{R}$, infinite or finite, if $\psi(x)$ vanishes at most once on $I$.
\\Note: This definition is not a negation of the definition of oscillating.
\end{definition}
\begin{theorem}
If the discriminant $D(x)\geq0$ of \eqref{ODE3} on a finite or infinite interval $I$, then a non-trivial solution $\psi(x)$ of \eqref{ODE3} does not oscillate on $I$
\end{theorem}
\begin{proof}
If $D(x)=0$ then $Q(x)=0$ in \eqref{nODE}. Therefore a solution to \eqref{nODE} has the form $\phi(x)=Ax+B$ where $A$ and $B$ are constants. This implies that any non-trivial solution to \eqref{ODE3} has the form  $\psi(x) = \mathrm{e}^{\frac{1}{2}\int{b\ \mathrm{d}x}}\left(Ax+B\right)$ via inverse transformation from normal form. Then $\psi(x)$ will vanish only when $Ax+B$ vanishes, which is at most once.\\

\noindent Let $\psi(x)$ be a nontrivial solution to \eqref{ODE3}. If $D(x)>0$ on $I$ then $Q(x)<0$, from \eqref{nODE}, on $I$ . By comparision with the equation $u''=0$, any solution, $\phi(x)$, to $u''+Q(x)u=0$ has at most one zero on $I$  since if $\phi(x)$ had more than one zero, then any solution, $u(x)=Ax+B$, of $u''=0$ must vanish between the zeros of $\phi(x)$ (Sturm's Comparison Theorem). Thich is impossible since there is always a linear function which fails this condition. By inverse transformation there exists a solution, $\phi(x)$, to \eqref{nODE} such that $\psi(x) = \mathrm{e}^{\frac{1}{2}\int{b\ \mathrm{d}x}}\phi(x)$, implying $\psi(x)$ vanishes on $I$ whenever $\phi(x)$ vanishes on $I$, which is at most once.\\
\qedhere
\end{proof}
\noindent To give an example, consider the modified Bessel differential equation,
\begin{equation}
y''+\frac{1}{x}y'+\left(1+\frac{n^2}{x^2}\right)y= 0 \notag
\end{equation}
Then $D(x)= \frac{4n^{2}-1}{x^2} + 4>0$ on $I=(-\infty,\infty)$ for $n\geq\frac{1}{2}$ \\\\
By the above theorem, solutions to this differential equation do not oscillate on $\mathbb{R}$. In particular, modified Bessel functions of the first kind, with $n\geq1$, do not oscillate on $\mathbb{R}$. This is shown in the graph below.
\vspace{5mm}\\
\centerline{\includegraphics[scale=.2]{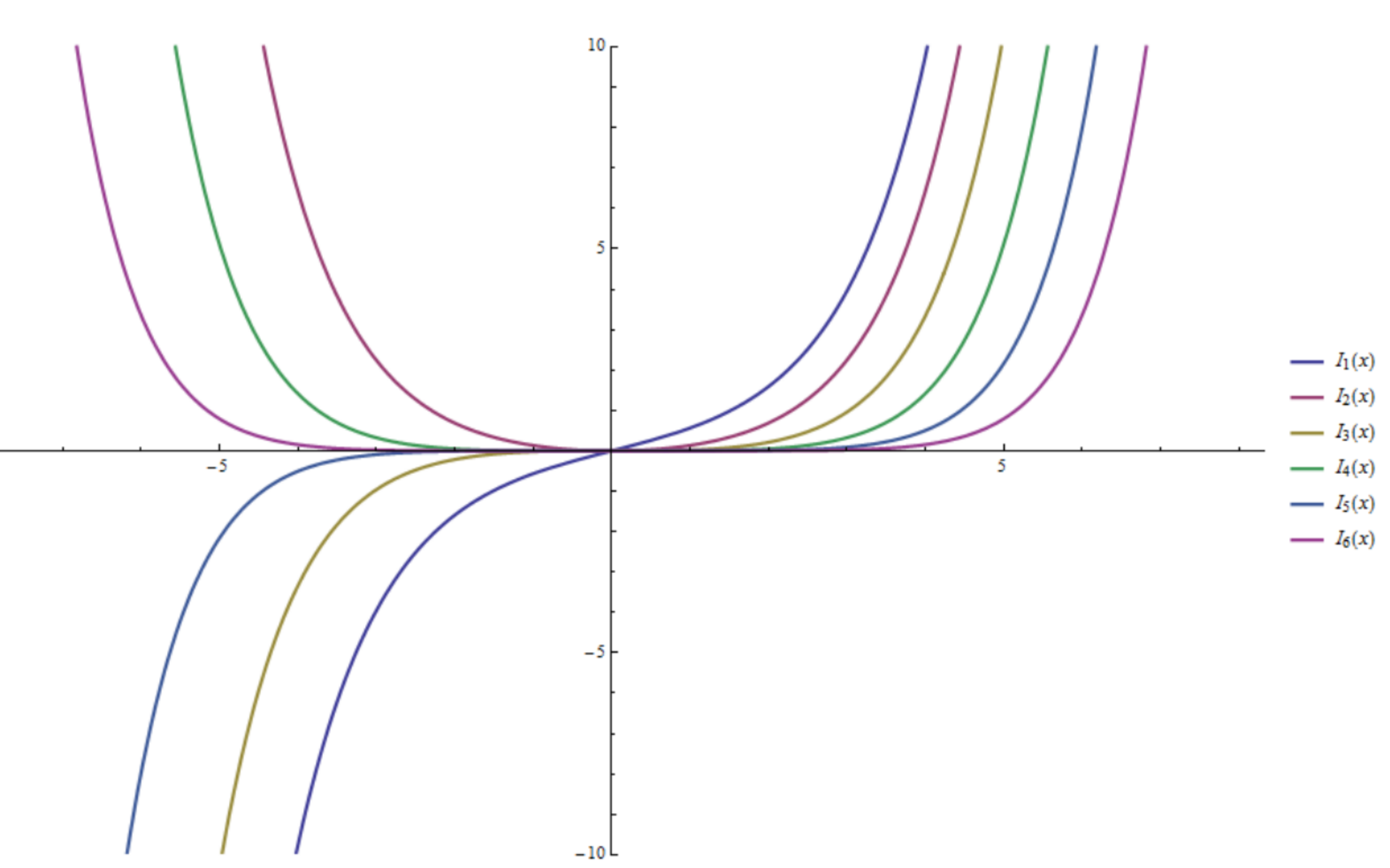}}\\
\vspace{1mm}\\
\begin{theorem}
If the discriminant $D(x)<\lambda$ for some $\lambda<0$ on an infinite interval $I$, then a non-trivial solution $\psi(x)$ of \eqref{ODE3}  oscillates on $I$.
\end{theorem}
\noindent This is simply a reformulation of Theorem 2.2, but using the discriminant as the focal point.\\\\
\noindent Notice that $D= b^{2}-4c$ in the case that $b(x)$ and $c(x)$ are constant functions, since $b'(x)=0$. This is the same definition as the original discriminant used for the constant coefficient case.
\section{Future Research}

The discriminant test yields sufficient criteria for oscillations on infinite intervals, and non-oscillations on infinite and finite interval. Beyond this, the discriminant seems to predict oscillations (in the sense of successive maxima and minima) on finite intervals as well. Before the topic of oscillations on finite intervals can be treated properly, one must define what it means to oscillate on a finite interval. Does a 10th order polynomial oscillate on an intercal that contains all its maxima and minima, but not on one that contain only two maxima and minima? Does a sinusoidal function oscillate on the interval $[0,\pi]$? These are the types of questions one comes across when trying to formulate a proper definition for oscillating; one that is mathematically rigorous, yet still agrees with one's intuition of what it means to oscillate.\\\\
\newpage
\noindent Future research will include properly defining oscillations on finite intervals, with sufficient and possibly neccessary conditions for oscillation based on the simple test of the discriminant. The discriminant test will then be extended to the cases of non-linear second order differential equations, and applied to differential equations of physical systems to quickly determine whether or not solutions will oscillate or not. The field of functional analysis, and more specifically oscillation theory, can give deep insight into the oscillatory behavior of solutions to second order equations, but involve very difficult and advanced methods of mathematics. The discriminant gives the basic behavior of the system, using only the knowledge prior of an undergraduate course in differential equations. This makes the discriminant another fast and effective tool for extracting the basic properties of solutions to second order equations.  Once one determines this basic oscillatory behavior, further analysis can be done using more advanced methods described above.

\enddocument